\documentclass[final,11pt,leqno]{amsart}

\usepackage{amsmath,amssymb,amscd,amsthm}
\usepackage{latexsym}
\usepackage{graphicx}

\usepackage{showkeys}

\newtheorem{theorem}{Theorem}
\newtheorem{lemma}{Lemma}
\newtheorem{proposition}{Proposition}

\newtheorem{definition}{Definition}
\newtheorem{corollary}{Corollary}

\newtheorem{claim}{Claim}

\newcommand{\f}[2]{\frac{#1}{#2}}
\newcommand{\dpr}[2]{\langle #1,#2 \rangle}

\newcommand{\al}{\alpha}

\newcommand{\de}{\delta}

\newcommand{\ve}{\varepsilon}

\newcommand{\ka}{\kappa}
\newcommand{\la}{\lambda}

\newcommand{\si}{\sigma}

\newcommand{\vp}{\varphi}

\newcommand{\rone}{\mathbf R^1}

\newcommand{\cl}{\mathcal L}

\newcommand{\p}{\partial}

\newcommand{\beq}{\begin{equation}}
\newcommand{\eeq}{\end{equation}}
\newcommand{\beqna}{\begin{eqnarray*}}
\newcommand{\eeqna}{\end{eqnarray*}}
\newcommand{\beqn}{\begin{equation*}}
\newcommand{\eeqn}{\end{equation*}}
\newcommand{\bp}{\begin{proof}}
\newcommand{\ep}{\end{proof}}
\newcommand{\bprop}{\begin{proposition}}
\newcommand{\eprop}{\end{proposition}}
\newcommand{\bt}{\begin{theorem}}
\newcommand{\et}{\end{theorem}}
\newcommand{\bex}{\begin{Example}}
\newcommand{\eex}{\end{Example}}
\newcommand{\bc}{\begin{corollary}}
\newcommand{\ec}{\end{corollary}}
\newcommand{\bcl}{\begin{claim}}
\newcommand{\ecl}{\end{claim}}
\newcommand{\bl}{\begin{lemma}}
\newcommand{\el}{\end{lemma}}

\begin{document}

\title
[Stability of traveling waves in the `abc' system]
{Spectral stability   for  subsonic traveling pulses of the Boussinesq `abc' system }

 \author{Sevdzhan Hakkaev}
\author{Milena Stanislavova}
\author{Atanas Stefanov}

\address{Sevdzhan Hakkaev
Faculty of Mathematics and Informatics, Shumen University, 9712
Shumen, Bulgaria}\email{shakkaev@fmi.shu-bg.net}
\address{Milena Stanislavova
Department of Mathematics, University of Kansas, 1460 Jayhawk
Boulevard,  Lawrence KS 66045--7523} \email{stanis@math.ku.edu}
\address{Atanas Stefanov
Department of Mathematics, University of Kansas, 1460 Jayhawk
Boulevard,  Lawrence KS 66045--7523}

\email{stefanov@math.ku.edu}

\thanks{Hakkaev supported in part by research grant  DDVU 02/91 of 2010 of the
Bulgarian Ministry of Education and Science.
Stanislavova  supported in part by NSF-DMS \# 0807894 and NSF-DMS \# 1211315.
Stefanov supported in part by NSF-DMS \# 0908802 .}

\date{\today}


\begin{abstract}
We consider the spectral stability of certain traveling wave solutions for the Boussinesq `abc' system. More precisely, we consider the explicit $sech^2(x)$ like solutions of the form \\ $(\vp(x-w t), \psi(x- w t)=(\vp, const.  \vp)$, exhibited by M. Chen, \cite{Ch1}, \cite{Ch2} and  we provide a complete rigorous characterization of the spectral stability in all cases for which    $a=c<0, b>0$. 
\end{abstract}
\
\maketitle
\maketitle

\section{Introduction and results}
 \subsection{The general Boussinesq `abcd' model} In this work,  we are concerned with the Boussinesq system
  \begin{equation}
  \label{b:1}
  \left|
  \begin{array}{l}
  \eta_t+ u_x+(\eta u)_x+ a u_{xxx}-b \eta_{xxt}=0\\
  u_t+\eta_x+u u_x+c \eta_{xxx}-d u_{xxt}=0.
  \end{array}
  \right.
  \end{equation}
The first formal derivation  for this system has appeared in the work of Bona-Chen-Saut, \cite{BCS}
to describe  the \nolinebreak  (essentially two dimensional) motion of small-amplitude long waves on the surface of an ideal fluid under the force of gravity. Here,  $\eta$ represents the vertical deviation of the free surface from its rest position, while  $u$ is the horizontal velocity at time $t$. In the case of zero surface tension $\tau=0$, the constants $a,b,c,d$ must satisfy in addition the consistency conditions $a+b=\f{1}{2}(\theta^2-1/3)$ and $c+d=\f{1}{2}(1-\theta^2)>0$. In the case of non-zero surface tension however, one only requires $a+b+c+d=\f{1}{3}-\tau$.
For this reason (as well as from the pure mathematical interest in the analysis of \eqref{b:1}), one may as well consider \eqref{b:1} for all values of the parameters.

 Systems of the form \eqref{b:1} have been the subject of intensive investigation over the last decade.  In particular,  the role of the parameters $a,b,c,d$ in the actual fluid models has been explored in great detail in the original paper \cite{BCS} and later in \cite{BCS1}. It was argued that only models in the form \eqref{b:1}, for which one has linear and nonlinear well-posedness  are physically relevant. We refer the reader to these two papers for further discussion and some precise conditions, under which one has such well-posedness theorems.

 Regarding explicit traveling wave solutions, Chen,  has considered various cases of interest in \cite{Ch1}, \cite{Ch2}. In fact, she has written down numerous traveling wave solutions (i.e.  in the form $(\eta, u)=(\vp(x-w t), \psi(x - w t))$, where in fact some of them are not necessarily  homoclinic to zero at $\pm \infty$. In a subsequent paper, \cite{Ch3}, Chen has also found new and explicit
  multi-pulsed traveling wave solutions.

In \cite{CCN}, Chen-Chen-Nguyen consider another relevant case, namely the BBM system, which  ($a=c=0, b=d=\f{1}{6}$). They construct periodic traveling wave solutions for the BBM case, as well as in more general situations. In \cite{AABCW}, the authors explore the existence theory for the the BBM system as well as its
relations to the single BBM equation.

 We wish to discuss another aspect of \eqref{b:1}, which is its Hamiltonian formulation. Since it is derived from the Euler equation by ignoring the effects of the dissipation, one generally expects such systems to exhibit a Hamiltonian structure. {\it This is however not generally the case, unless one imposes some further restrictions on the parameters}. Indeed, if $b=d$, one can easily check that
 $$
 H(\eta, u)=\int -c \eta_x^2-au_x^2+\eta^2+(1+\eta) u^2 dx
 $$
 Furthermore, $H(\eta, u)$ is positive definite only if $a, c<0$. From this  point of view, it looks natural to consider the case $b=d$ and $a,c<0$. In order to focus our discussion, we shall concentrate then on this version
\begin{equation}
  \label{1}
  \left|
  \begin{array}{l}
  \eta_t+ u_x+(\eta u)_x+ a u_{xxx}-b \eta_{xxt}=0\\
  u_t+\eta_x+u u_x+c \eta_{xxx}-b u_{xxt}=0.
  \end{array}
  \right.
  \end{equation}
We will refer to \eqref{1} as the Boussinesq `abc' system.
It is a standard practice that stable coherent structures, such as traveling pulses etc. are produced as constrained  minimizers of the corresponding (positive definite) Hamiltonians, with respect to a fixed conserved quantity.
In fact, this program has been mostly carried out, at least in the Hamiltonian cases, in a series of papers by Chen, Nguyen and Sun. More precisely, in \cite{CNS}, the authors have shown that traveling waves for \eqref{b:1} exist in the regime\footnote{which in particular requires  that $a+b+c+d<0$, corresponding to a  ``large'' surface tension $\tau>\f{1}{3}$} $b=d$,
 $a, c<0, ac>b^2$.  In addition, they have also shown stability of such waves in the   sense of a `set stability' of the set of  minimizers.  In the companion paper \cite{CNS1}, the authors have considered the general case $b=d>0$, $a,c<0$, which in particular allows for small surface tension. 
 
  The existence of a   traveling wave was proved for every speed
 $|w| \in (0, \min(1,\f{\sqrt{ac}}{b}))$.This is the so-called subsonic  regime. Finally, we point out to a recent work by Chen, Curtis, Deconinck, Lee and Nguyen, \cite{CCDLN} in which the authors study numerically various aspects of spectral stability/instability of some solitary waves of \eqref{b:1}, including the multipulsed solutions exhibited in \cite{Ch3}. In the same paper, the authors
 also study (numerically) the transverse stability/instability of the same waves, viewed as solutions to the two dimensional problem.

 The purpose of this paper is to study rigorously  the spectral stability of some explicit traveling waves in the regime $b=d>0$, $a, c<0$. This would be achieved via the use of the instabilities indices counting formulas of Kapitula, Kevrekidis and Sandstede, \cite{KKS1}, \cite{KKS2} and the subsequent refinement by Kapitula, Stefanov  \cite{KS}.

\subsection{The traveling wave solutions}
In this section, we follow almost verbatim the description of some explicit solutions of interest of \eqref{b:1}, given by Chen, \cite{Ch1}, see also the more detailed exposition of the same results in \cite{Ch2}. More precisely,
the solutions  of interest are traveling waves, that is in the form
$$
\eta=\vp(x-w t), \ \ u(x,t)=\psi(x-w t).
$$
A direct computation shows that if we require that the
pair $(\vp, \psi)$ vanishes at $\pm \infty$, then it satisfies the system
\begin{equation}
  \label{5}
  \left|
  \begin{array}{l}
   (1+c\p_x^2)\vp-w(1-b\p_x^2)\psi+\f{\psi^2}{2}=0\\
  -w(1-b\p_x^2)\vp+(1+a\p_x^2)\psi+\vp \psi=0.
  \end{array}
  \right.
  \end{equation}
The typical ansatz that one starts with, in order to simplify the system
\eqref{5} to a single equation  is $\psi=B\vp$. This has been worked out by Chen, \cite{Ch1}, \cite{Ch2}.  The following result is contained in the said papers. 
\newpage 
\begin{theorem}(Chen, \cite{Ch1}, \cite{Ch2})
\label{mchen}
Let the parameters $a,b,c$ in the system satisfy one of the following
\begin{enumerate}
\item $a+b\neq 0$, $p=\f{c+b}{a+b}>0$, $(p-1/2)((b-a)p-b)>0$
\item $a=c=-b$, $b>0$
\end{enumerate}
Then, there are the following (pair of) exact traveling
wave solutions (i.e. solutions of \eqref{5}) $(\vp(x- wt), \psi(x- wt))$, where
\begin{eqnarray*}
& & \vp(x)=\eta_0 sech^2(\la x) \\
& & \psi(x)=  B(\eta_0)\eta_0 sech^2(\la x)
\end{eqnarray*}
and
$$
w=w(\eta_0)= \pm \frac{3+2\eta_0}{\sqrt{3(3+\eta_0)}};\  \la= \f{1}{2} \sqrt{\f{2\eta_0}{3(a-b)+2b(\eta_0+3)}}; \ \ B(\eta_0)=\pm \sqrt{\f{3}{\eta_0+3}},
$$
and $\eta_0$ is a constant that satisfies
\begin{enumerate}
\item $\eta_0=\f{3(1-2p)}{2p}$ in Case $(1)$
\item $\eta_0>-3, \eta_0\neq 0$ in Case $(2)$.
\end{enumerate}
\end{theorem}

\subsection{Different notions of stability}
Before we state our results, we pause to discuss the various definitions of stability. First, one says that the solitary wave solution $(\vp_w, \psi_w)$ is orbitally stable, if for every $\ve>0$, there exists $\de>0$, so that whenever
$\|(f,g)-(\vp_w, \psi_w)\|_{X}<\de$, one has that the corresponding solutions $(\eta, u): (f,g)=(\eta, u)|_{t=0}$
$$
\sup_{t>0} \inf_{x_0}\|(\eta(x-x_0,t), u(x -x_0,t))-(\vp(x-w t), \psi(x - w t))\|_X<\ve.
$$
Note that we have not quite specified a space $X$, since this usually
depends on the particular problem at hand (and mostly on the available conserved quantities), but suffices to say that $X$  is usually chosen to be a natural energy space for the problem.
This notion of (nonlinear) stability has been of course successfully used to treat a great deal of important problems, due to the versatility of the classical Benjamin and Grillakis-Shatah-Strauss approaches. However, it looks like these methods are not readily applicable (if at all) to the Boussinesq `abc'  system. We encourage the interested reader to consult the discussion in \cite{CNS}, where a weaker, but related stability was established in the regime $ac>b^2$ and additional smallness assumption on the wave is required as well. This is why, one needs to develop an alternative approach to this important problem, which is one of the main goals of this work.

In this paper, we will concentrate on spectral stability. There is also (the closely related and almost equivalent) notion of linear stability, which we also mention below. In order to introduce the object of our study, as well as to motivate its relevance, let us perform a linearization of the nonlinear system \eqref{1}. Using  the ansatz
 $$
  \left|
  \begin{array}{l}
   \eta=\vp(x-w t)+v(t, x-w t)\\
  u=\psi(x-w t)+z(t, x-w t),
  \end{array}
  \right.
 $$
in \eqref{1} and ignoring all quadratic terms in the form $O(v^2), O(v z), O(z^2)$ leads to the following linearized problem
$$
(1-b\p_x^2)\left(\begin{array}{c}
v \\ z
\end{array}\right)_t= -\p_x \left(\begin{array}{c c}
0 & 1  \\ 1 & 0
\end{array}\right)\left(\begin{array}{cc}
1+c\p_x^2 & b w \p_x^2 +\psi-w \\
b w \p_x^2 +\psi-w & 1+a \p_x^2 +\vp
\end{array}\right)
$$
Letting
\begin{equation}
\label{110}
L=\left(\begin{array}{cc}
1+c\p_x^2 & b w \p_x^2 +\psi-w \\
b w \p_x^2 +\psi-w & 1+a \p_x^2 +\vp
\end{array}\right), J= -\p_x (1-b\p_x^2)^{-1}\left(\begin{array}{c c}
0 & 1  \\ 1 & 0
\end{array}\right)
\end{equation}
the linearized problem that we need to consider may be written in the form
\begin{equation}
\label{90}
u_t=J L u
\end{equation}
Note that in the whole line context, $L$ is a self-adjoint operator, when considered with the natural domain $D(L)=H^2(\rone)\times H^2(\rone) $.  Letting  $H:= J L$, we see that the problem \eqref{90} is in the form $u_t=H u$. The study of linear problems in this form is at the basis of
the deep theory of   $C_0$ semigroups. Informally, if the Cauchy problem $u_t=H u$ has global solutions for all smooth and decaying data, we say that $H$ generates a $C_0$ semigroup $\{T(t)\}_{t>0}$ via the exponential map $T(t)=e^{t H}$. Furthermore, we say that we have {\it linear stability} for the linearized problem $u_t=H u$, whenever the growth rate of the semigroup is zero or equivalently $\lim_{t\to \infty} e^{-\de t} \|T(t) f\|=0$ for all $\de>0$ and for all sufficiently smooth and decaying functions
$f$. Finally, we say that the system is {\it spectrally stable}, if $\si(H)\subset \{z: \Re z\leq 0\}$. It is well-known that if $H$ generates a $C_0$ semigroup, then linear stability implies spectral stability, but not vice versa. Nevertheless, the two notions are very closely related and in many cases (including the ones under consideration), they are indeed equivalent. For the purposes of  a formal definition, we proceed as follows
\begin{definition}
\label{defi:1}
We say that the problem \eqref{90} is unstable, if there is $\mathbf{f}\in H^2(\rone)\times H^2(\rone)$ and $\la: \Re \la>0$, so that
\begin{equation}
\label{100}
J L \mathbf{f}=\la \mathbf{f}.
\end{equation}
Otherwise, the problem \eqref{90} is stable. That is, stability is equivalent to the absence of solutions  of \eqref{100} with $\la: \Re\la>0$.
\end{definition}
\subsection{Main results}
We are now ready to state our results. We chose to split them in two cases, just as in Theorem \ref{mchen}. For the case $a=c=-b, b>0$, we have
\begin{theorem}
\label{theo:1}
Let $a=c=-b, b>0$. Then, the traveling wave solutions of the `abc' system
\begin{equation}
\label{waves}
\left(\eta_0 sech^2\left(\f{x - w t}{2\sqrt{b}}\right),  \pm \eta_0
\sqrt{\f{3}{\eta_0+3}}  sech^2\left(\f{x-w t}{2\sqrt{b}}\right)\right)
\end{equation}
with speed $w=\pm \frac{3+2\eta_0}{\sqrt{3(3+\eta_0)}}$
are stable, for all $\eta_0: 
\eta_0\in (-\f{9}{4},0). 
$
Equivalently, all waves in \eqref{waves} are stable, for all  speeds $|w|<1$. 
\end{theorem}
\noindent Note that $|w|<1$ is equivalent to $\eta_0\in (-\f{9}{4},0)$, so we assume this henceforth. 
In the remaining case, we assume only $a=c<0, b=d>0$, but observe that in this case,   Theorem \ref{mchen}  requires that $\eta_0=-3/2, w=0$, that is the traveling waves become standing waves. 
 \begin{theorem}
 \label{theo:20}
 Let $a=c<0, b=d>0$.
 Then, the  standing wave solutions of the Boussinesq system 
 $$
 \vp(x)=-\f{3}{2} sech^2\left(\f{x}{2\sqrt{-a}}\right), \psi(x)= \pm \f{3}{\sqrt{2}} 
 sech^2\left(\f{x}{2\sqrt{-a}}\right)
 $$
 are spectrally stable \underline{if and only if} 
\begin{equation}
\label{1000}
 \dpr{(a\p_x^2+1-\vp)^{-1}(\vp- b \vp'')}{(\vp- b \vp'')}\leq 8 \sqrt{-a}
  \left( \f{9}{2}+\f{12}{5} \f{b}{|a|} - 
 \f{3}{10} \f{b^2}{a^2}\right).
 \end{equation}
 In particular, the condition \eqref{1000} holds ( and thus the waves are spectrally stable),  whenever 
 $$
0\leq \f{b}{-a}< 8.00163, 
$$
On the other hand, the condition \eqref{1000} fails ( and thus the waves are spectrally unstable), if 
$$
\f{b}{-a}>   8.82864. 
$$
 \end{theorem}
{\bf Remark: }
  Note that while, we  cannot explicitly compute   the value $(a\p_x^2+1-\vp)^{-1}(\vp- b \vp'')$ in \eqref{1000}, we obtain   estimates, which  imply some pretty good results for the stability/instability intervals. One can in fact push this further to narrow the gap between the stability and instability regions, predicted by \eqref{1000}. This can be done in principle with any degree of accuracy, but it increases the complexity  the argument.

\section{Preliminaries}
 In this section, we collect some preliminary results, which will be useful in the sequel.

\subsection{Some spectral properties of $L$}
We shall need some   spectral information about the operator $L$. We collect the results in the following
\begin{proposition}
\label{prop:10}
Let $a,c<0$ and  $w: 0\leq |w|<\min\left(1, \f{\sqrt{a c}}{|b|}\right)$. Then, the self-adjoint operator $L$ has the following spectral properties
\begin{itemize}
\item
  Then the operator $L$ has an eigenvalue at zero, with an eigenvector
$\left(\begin{array}{c} \vp' \\ \psi' \end{array} \right)$.
\item There is $\ka>0$, so that the essential spectrum is in $\si_{ess}(L)\subset [\ka, \infty)$.
\end{itemize}
\end{proposition}
\begin{proof}
The first property is easy to establish, this is the usual eigenvalue at zero generated by translational invariance. For the proof,  all one needs to do is take a spatial derivative in the defining system \eqref{5}, whence $L\left(\begin{array}{c} \vp' \\ \psi' \end{array} \right)= \left(\begin{array}{c} 0 \\  0 \end{array} \right)$.

Regarding the essential spectrum, we reduce matters to the Weyl's theorem (using the vanishing of the waves at $\pm \infty$), which ensures that
$$
\si_{ess.}(L)=\si_{ess.}(L_0)=\si[\left(\begin{array}{cc}
1+c\p_x^2 & b w \p_x^2 -w \\
b w \p_x^2 -w & 1+a \p_x^2
\end{array}\right)]
$$
That is, it remains to check that the matrix differential operator $L_0>\ka$.
By Fourier transforming $L_0$, it will suffice to check that the matrix
$$
L_0(\xi)=\left(\begin{array}{cc}
1-c\xi^2 & -w(b\xi^2 +1) \\
-w(b\xi^2 +1)& 1-a \xi^2
\end{array}\right)
$$
is positive definite for all $\xi\in \rone$. Since $1-c\xi^2\geq 1$, it will suffice to check that the determinant has a  positive minimum over $\xi\in \rone$.
We have
$$
det(L_0(\xi))= \xi^4(ac-b^2 w^2)+\xi^2(-a-c-2b w^2)+(1-w^2)\geq (1-w^2)+
2\xi^2(\sqrt{a c}- |b| w^2),
$$
where in the last inequality, we have used  $-a-c\geq 2 \sqrt{a c}$. The strict
positivity follows by observing that $\sqrt{ac}\geq |b| w  \geq |b|w^2$, since $w<1$.
\end{proof}

\subsection{Instability index count}  In this section, we introduce the instability indices counting formulas, which in many cases of interest can in fact be used to determine accurately both stability and instability regimes for the waves under consideration. As we have mentioned above, this theory has been under development for some time, see \cite{Mad}, \cite{Kap}, \cite{Pel}, but we use a recent formulation due to Kapitula-Kevrekidis and Sandstede (KKS), \cite{KKS1} (see also \cite{KKS2}). In fact, even the (KKS) index count formula is not directly applicable\footnote{due to a crucial assumption for invertibility of the skew-symmetric operator $J$, which is not satisfied for $\p_x$ acting on $\rone$} to the
   problem of \eqref{90}, which is why Kapitula and Stefanov, \cite{KS} have found an approach, based on the KKS of the theory, which covers this situation.
   In order to simplify the exposition, we will restrict to a corollary of the main result in \cite{KS}.
More precisely, a the stability problem in the form is considered in the form
\begin{equation}
\label{c:5}
\partial_x\cl u= \lambda u,
\end{equation}
where $\cl$ is a self-adjoint linear differential operator with domain
$D(\cl)=H^s(\rone)$ for some $s$. It is assumed that for the
operator $\cl$,
\begin{enumerate}
\item there are $n(\cl)=N<+\infty$ negative eigenvalues\footnote{We will henceforth denote by $n(M)$ the number of  negative eigenvalues (counting multiplicities)   of
a self-adjoint operator $M$} (counting multiplicity), so that each of the corresponding eigenvectors $\{f_j\}_{j=1}^N$ belong to $H^{1/2}(\rone)$.
\item there is a $\kappa>0$ such that
$\sigma_{ess}(\cl)\subset[\kappa^2,+\infty)$
\item $\dim[\ker(\cl)]=1$, $\ker(\cl)= span\{\psi_0\}$,
$\psi_0$ real-valued function, $\psi_0\in H^{\infty}(\rone)\cap \dot{H}^{-1}(\rone)$.
\end{enumerate}
  Here, $\dot{H}^{-1}(\rone)$ is the homogeneous Sobolev space, defined via the   norm
  $$
  \|u\|_{\dot{H}^{-1}(\rone)}:=\left( \int_{\rone} \f{|\hat{u}(\xi)|^2}{|\xi|^2} d\xi\right)^{1/2}.
  $$
  or equivalently, $u=\p_x z$ in sense of distributions, where $z\in L^2$  and
  $\|u\|_{\dot{H}^{-1}(\rone)}:=\|z\|_{L^2}$. In that case, we have
  \begin{theorem}(Theorem 3.5, \cite{KS})
  \label{t:index}
For the eigenvalue problem
\begin{equation}
\label{c:10}
\partial_x\cl u=\lambda u,\quad u\in L^2(\rone),
\end{equation}
where the self-adjoint operator $\cl$ satisfies $D(\cl)=H^s(\rone)$ for some
$s>0$, assume that
\[
\langle\cl^{-1}\partial_x^{-1}\psi_0,\partial_x^{-1}\psi_0\rangle\neq0.
\]
Then, the number of  solutions of \eqref{c:5}, $n_{unstable}(\cl)$,
with  $\la:\Re \la>0$ satisfies\footnote{here $\p_x^{-1} \psi_0$ is {\bf any} $L^2$ function $f$, so that $\psi_0=\p_x f$ in distributional sense}
\begin{equation}
\label{120}
0\leq n_{unstable}(\p_x \cl)= n(\cl)-n\left(\langle\cl^{-1}\partial_x^{-1}\psi_0,\partial_x^{-1}\psi_0\rangle\right)\mod 2.
\end{equation}
\end{theorem}
 Of course, our eigenvalue problem \eqref{100} does not immediately
 fit the form of Theorem \ref{t:index}. First, Theorem \ref{t:index} applies for scalar-valued operators $\cl$, while  we need to deal with vector-valued operators. This is a minor issue and in fact, one sees easily that the arguments   in \cite{KS} carry over easily in the case, where $\cl$ is a vector-valued self-adjoint operator as well. A second, more substantive issue is that the form of \eqref{100} is not quite the one in \eqref{c:10}. Namely, we have that the operator $J$, while still skew-symmetric is not equal to $\p_x$.

 In order to fix that, we need to recast the eigenvalue problem \eqref{100} in a slightly different form. Indeed, letting $\mathbf{f}=(1-b\p_x^2)^{-1/2}\mathbf{g}$ and taking $(1-b\p_x^2)^{1/2}$ on both sides of \eqref{100},  we may rewrite it as follows
 $$
 -\p_x   \left(\begin{array}{c c}
0 & 1  \\ 1 & 0
\end{array}\right)(1-b\p_x^2)^{-1/2} L (1-b\p_x^2)^{-1/2} \mathbf{g}=\la \mathbf{g}.
 $$
  If we now introduce
  $$
  \tilde{J}:= -\p_x   \left(\begin{array}{c c}
0 & 1  \\ 1 & 0
\end{array}\right); \ \ \tilde{L}:=(1-b\p_x^2)^{-1/2} L (1-b\p_x^2)^{-1/2},
  $$
  we easily see that $\tilde{J}$ is still anti-symmetric, $\tilde{L}$ is
  self-adjoint and we have managed to represent the eigenvalue problem in the form $\tilde{J} \tilde{L} \mathbf{g}=\la \mathbf{g}$. Note that the operator $\tilde{J}$ is very similar to $\p_x$, except for the action of the invertible symmetric operator $\left(\begin{array}{c c}
0 & 1  \\ 1 & 0
\end{array}\right)$   on it. It is not hard to see that the result of Theorem \ref{t:index} applies to it (while it still fails the standard conditions of the KKS theory, due to the non-invertibility of $\tilde{J}$). Note that one needs to replace $\p_x^{-1}$ by $\tilde{J}^{-1}$ in the formula \eqref{120}. Furthermore, the number of unstable modes for the two systems ($J L$ and $\tilde{J} \tilde{L}$) is clearly the same, due to the simple transformation $(1-b\p_x^2)^{-1/2}$ connecting the corresponding eigenfunctions.

Thus, {\it if we can verify the conditions under which Theorem \ref{t:index} applies}, we get the stability index formula
\begin{equation}
\label{130}
n_{unstable}(J L)=n_{unstable}(\tilde{J} \tilde{L})=n(\tilde{L})-
n(\dpr{\tilde{L}^{-1} \tilde{J}^{-1}\psi_0}{\tilde{J}^{-1}\psi_0}) \mod 2.
\end{equation}
Since by Proposition \ref{prop:10}, $L\left(\begin{array}{c}\vp'\\ \psi'\end{array}\right)=0$, we conclude that $\tilde{L}[(1-b \p_x^2)^{1/2} \left(\begin{array}{c}\vp'\\ \psi'\end{array}\right)]=0$. It follows that
$\psi_0=\p_x (1-b \p_x^2)^{1/2} \left(\begin{array}{c}\vp \\ \psi\end{array}\right)$ and
\begin{eqnarray*}
 \dpr{\tilde{L}^{-1} \tilde{J}^{-1}\psi_0}{\tilde{J}^{-1}\psi_0} &=&
\dpr{ L^{-1} [(1-b \p_x^2) \left(\begin{array}{c c}
0 & 1  \\ 1 & 0
\end{array}\right)   \left(\begin{array}{c}\vp \\ \psi\end{array}\right)]}{ (1-b \p_x^2) \left(\begin{array}{c c}
0 & 1  \\ 1 & 0
\end{array}\right)  \left(\begin{array}{c}\vp \\ \psi\end{array}\right)}=\\
&=& \dpr{L^{-1}[(1-b \p_x^2) \left(\begin{array}{c}\psi \\ \vp \end{array}\right)]}{(1-b \p_x^2) \left(\begin{array}{c}\psi \\ \vp \end{array}\right)}
\end{eqnarray*}
Thus, we conclude that we will have established spectral stability for \eqref{100},
if we can verify the conditions $(1), (2), (3)$ of Theorem \ref{t:index} for the operator $\tilde{L}$,
 $n(\tilde{L})=1$ and
 \begin{equation}
 \label{150}
 \dpr{L^{-1}[(1-b \p_x^2) \left(\begin{array}{c}\psi \\ \vp \end{array}\right)]}{(1-b \p_x^2) \left(\begin{array}{c}\psi \\ \vp \end{array}\right)}<0.
 \end{equation}
 and instability otherwise. 

Concretely, we will verify the conditions on $\tilde{L}$ in Proposition \ref{prop:25} below, after which, we   compute the quantity in \eqref{150} in Proposition
\ref{prop:30}.
\begin{proposition}
\label{prop:25}
The self-adjoint operator $\tilde{L}=(1-b \p_x^2)^{-1/2}L (1-b \p_x^2)^{-1/2}$ satisfies
\begin{enumerate}
 \item $\si_{ess.}(\tilde{L})\subset [\ka, \infty)$ for some positive $\ka$.
\item  $n(\tilde{L})=1$.
 \item $Ker(\tilde{L})=span\{(1-b \p_x^2)^{1/2}\left(\begin{array}{c} \vp'\\ \psi'\end{array}\right)\}$.
\end{enumerate}
in the following cases 
\nopagebreak 
\begin{itemize}
\item $a=c=-b, b>0$, $B=\pm \sqrt{\f{3}{3+\eta_0}}$, $w=\pm \f{3+2\eta_0}{\sqrt{3(3+\eta_0)}}, 
\eta_0\in (-\f{9}{4}, 0)$.
\item $a=c<0$, $b>0$, $w=0, B=\pm \sqrt{2}$. 
\end{itemize}
\end{proposition}

\begin{proposition}
\label{prop:30}
Regarding the instability index, we have 
\begin{itemize}
\item For $a=c=-b, b>0$, $w=\pm \f{3+2\eta_0}{\sqrt{3(3+\eta_0)}}$, $B(\eta_0)=\pm \sqrt{\f{3}{3+\eta_0}}$,   and for all $\eta_0\in (-\f{9}{4},0)$, 
$$
\dpr{L^{-1}[(1-b \p_x^2) \left(\begin{array}{c}\psi \\ \vp \end{array}\right)]}{(1-b \p_x^2) \left(\begin{array}{c}\psi \\ \vp \end{array}\right)} < 0  
$$
\item For $a=c<0$, $b>0$, $w=0$, $B=\pm \sqrt{2}$, 
\begin{eqnarray*}
& & \dpr{L^{-1}[(1-b \p_x^2) \left(\begin{array}{c}\psi \\ \vp \end{array}\right)]}{(1-b \p_x^2) \left(\begin{array}{c}\psi \\ \vp \end{array}\right)} =  \\
&=&  \f{1}{3}\left(8 \sqrt{-a}
  \left( -\f{9}{2}-\f{12}{5} \f{b}{|a|}+ 
 \f{3}{10} \f{b^2}{a^2}\right)+ \dpr{(a\p_x^2+1-\vp)^{-1} f}{f}\right)
\end{eqnarray*} 
 In particular, 
 \begin{eqnarray*}
 & & \dpr{L^{-1}[(1-b \p_x^2) \left(\begin{array}{c}\psi \\ \vp \end{array}\right)]}{(1-b \p_x^2) \left(\begin{array}{c}\psi \\ \vp \end{array}\right)}<0,  \ \ 0<\f{b}{-a}<8.00163,\\
& &  \dpr{L^{-1}[(1-b \p_x^2) \left(\begin{array}{c}\psi \\ \vp \end{array}\right)]}{(1-b \p_x^2) \left(\begin{array}{c}\psi \\ \vp \end{array}\right)}>0,  \  \ 
 \f{b}{-a}>8.82864.
 \end{eqnarray*}
\end{itemize}
\end{proposition}
Theorem \ref{theo:1} follows by virtue of 
Proposition \ref{prop:25} and Proposition \ref{prop:30}. 
Thus, it remains to prove  these two.

\section{Proof of Proposition \ref{prop:25}}
We start with the gap condition for $\si_{ess.}(\tilde{L})$ stated in
Proposition \ref{prop:25}.
 \subsection{$\tilde{L}$ is strictly positive}
The idea is contained in Proposition \ref{prop:10}. Write
$$
\tilde{L}=(1-b \p_x^2)^{-1/2}L (1-b \p_x^2)^{-1/2}=(1-b \p_x^2)^{-1/2}L_0 (1-b \p_x^2)^{-1/2}+(1-b \p_x^2)^{-1/2}(L-L_0)(1-b \p_x^2)^{-1/2},
$$
where $L-L_0$ is a multiplication by smooth and decaying potential. It is also not hard to see that $(1-b\p_x^2)^{-1/2}$ is given by a convolution     kernel
$K: K(x)=\int_{-\infty}^\infty \f{e^{2\pi i x \xi}}{\sqrt{1+4\pi^2 b\xi^2}} d\xi$, which decays faster than polynomial at $\pm \infty$.
 It follows that the operator $(1-b \p_x^2)^{-1/2}(L-L_0)(1-b \p_x^2)^{-1/2}$ is a compact operator on $L^2(\rone)$ and hence  By Weyl's theorem
 $$
 \si_{ess.}(\tilde{L})= \si_{ess.}((1-b \p_x^2)^{-1/2}L_0 (1-b \p_x^2)^{-1/2})=\si((1-b \p_x^2)^{-1/2}L_0 (1-b \p_x^2)^{-1/2})
 $$
 Thus, as we have explained in the proof of Proposition \ref{prop:10},
 it will suffice to check that the matrix
 $$
 (1+4\pi^2 b \xi^2)^{-1/2} L_0(\xi) (1+4\pi^2 b \xi^2)^{-1/2}
 $$
 is positive definite. But since $L_0(\xi)$ is positive definite, the result follows. Note that this only shows that $\si_{ess.}(\tilde{L})\geq 0$. Since we need to show an actual gap between $\si_{ess.}(\tilde{L})$ and zero, it suffices to observe (by the arguments in Proposition \ref{prop:10}) that the eigenvalues of $L_0(\xi)$ have the rate of $O(\xi^2)$ for large $\xi$, which implies that the positive eigenvalues of $(1+4\pi^2 b \xi^2)^{-1/2} L_0(\xi) (1+4\pi^2 b \xi^2)^{-1/2}$ have the rate of $O(1)$.

\subsection{The negative eigenvalue and the zero eigenvalue are both simple}

We now pass to the harder task of establishing the existence and simplicity of a negative eigenvalue for $\tilde{L}$ as well as the simplicity of the zero eigenvalue. Note that as we have already observed $L\left(\begin{array}{c}\vp' \\ \psi' \end{array}\right)=0$. It follows that
$$
\tilde{L}[(1-b\p_x^2)^{1/2} \left(\begin{array}{c} \vp'\\ \psi'\end{array}\right)]=
(1-b\p_x^2)^{-1/2} [L    \left(\begin{array}{c} \vp'\\ \psi'\end{array}\right)]=0.
$$
Thus, we have already identified one element of $Ker(\tilde{L})$, but it still remains to prove that $dim(Ker(\tilde{L}))=1$, in addition to the existence and the simplicity of the negative eigenvalue of $\tilde{L}$.

  Next, we find it convenient to introduce the following notation for the eigenvalues of a self-adjoint operator $\cl$. Indeed, assume that $\cl=\cl^*$ is bounded from below, $\cl\geq -c$,   we order\footnote{We follow the standard convention  that if an equality appears multiple  times in the sequence of eigenvalues, that signifies that eigenvalue has the same multiplicity}   the eigenvalues  as follows
$$
\inf spec(\cl)=\la_0(\cl) \leq \la_1(\cl)\leq \ldots.
$$
Recall also the following max min principle, due to Courant
$$
\la_0(\cl)=\inf_{\|f\|=1} \dpr{\cl f}{f}, \ \ \la_1(\cl)=\sup_{g\neq 0}
\inf_{\|f\|=1, f\perp g} \dpr{\cl f}{f}, \la_2(\cl)=\sup_{g_1, g_2: g_1\neq a g_2}
\inf_{\|f\|=1, f\perp span[g_1, g_2]} \dpr{\cl f}{f}.
$$
Clearly, our claims can be recast   in the more compact form
\begin{equation}
\label{15}
\la_0(\tilde{L})<0=\la_1(\tilde{L})<\la_2(\tilde{L}).
\end{equation}
 matters from $\tilde{L}$ to standard second order differential  operators, like $L$.
 \begin{lemma}
 \label{le:100}
 Let $a,c<0, b>0$ and  $w: 0\leq |w|<\min\left(1, \f{\sqrt{a c}}{|b|}\right)$. Then
 \begin{itemize}
 \item  all eigenvectors of $L$ from \eqref{110}, corresponding to non-positive  eigenvalues, belong to \\
 $H^\infty(\rone)=\cap_{l=1}^\infty H^l(\rone) $.
 \item If $\cl$ is any bounded from below
 self-adjoint operator, for which  \\ $\la_0(\cl)<0=\la_1(\cl)<\la_2(\cl)$,
 and $S$ is a bounded invertible   operator, then
 $$
 \la_0(S^*\cl S)<0=\la_1(S^*\cl S)<\la_2(S^*\cl S).
 $$
 \item If $L$ has the property $\la_0(L)<0=\la_1(L)<\la_2(L)$, then so does \\
 $\tilde{L}=(1-b\p_x^2)^{-1/2} L (1-b\p_x^2)^{-1/2}$. That is, \eqref{15} holds.
 \end{itemize}
 \end{lemma}
 \begin{proof}(Lemma \ref{le:100})

 Take the eigenvector $\mathbf{f}$, corresponding to $-a^2, a\geq 0$, i.e. $L \mathbf{f}=-a^2 \mathbf{f}$. As observed in the proof of Proposition \ref{prop:10}, we can represent $L=L_0+\mathbf{V}$, where $\mathbf{V}$ is smooth and decaying matrix potential. In addition, recall $L_0\geq \ka$, hence $L_0+a^2\geq \ka Id$ and hence invertible.  It follows that the eigenvalue problem at $-a^2$ can be rewritten  in the equivalent form
 $$
\mathbf{f}=-(L_0+a^2)^{-1}[\mathbf{V}\mathbf{f}]
 $$
 Clearly, $(L_0+a^2)^{-1}:L^2\to H^2$, whence we get immediately that $f\in H^2$, if $f\in L^2$. Bootstrapping this argument (recall $\mathbf{V}\in C^\infty$)
  yields $f\in H^4, H^6$ etc. In the end, $\mathbf{f}\in H^\infty$.

 Next, we have
  $$
  \la_0(S^*\cl S)
 =\inf_{f:\|f\|=1} \dpr{ S^*\cl S f}{f}=\inf_{f\neq 0} \f{\dpr{\cl S f}{S f}}{\|f\|^2}= \inf_{g\neq 0} \f{\dpr{\cl g}{g}}{\|S^{-1} g\|^2} <0,
$$
since $\la_0(\cl)=\inf_{g:\|g\|=1} \dpr{\cl g}{g}<0$. Since $\la_1(\cl)=0$, it follows that there is $h$, so that  \\
$
\inf_{g\perp h}  \dpr{\cl g}{g} \geq 0.
$
Thus,
$$
\la_1(S^*\cl S)\geq \inf_{f\perp S h} \f{\dpr{S^*\cl S f}{f}}{\|f\|^2}=
\inf_{g\perp  h} \f{\dpr{\cl g}{g}}{\|S^{-1} g\|^2}\geq 0.
$$
 Since $0$ is still an eigenvalue for $\cl$ with say eigenvector $\chi$, it follows that $S^{-1} \chi$ is an eigenvector to $S^*\cl S$, so $0$ is also an eigenvalue for
 $S^*\cl S$ and hence $\la_1(S^*\cl S)=0$.

 Regarding  $\la_2(S^* \cl S)$, we already know that
 $\la_2(S^* \cl S)>\la_1(S^* \cl S)=0$. Assuming the contrary would mean that $\la_2(S^* \cl S)=0$, that is $0$ is a double eigenvalue for $S^* \cl S$, say with linearly independent eigenvectors $f_1, f_2$. From this and the invertibility of $S$, it follows that $S^{-1} f_1, S^{-1} f_2$ are two linearly independent vectors in $Ker(L)$, a contradiction with the assumption that $0$ is a simple eigenvalue for $L$.

 The result regarding $(1-b\p_x^2)^{-1/2} L (1-b\p_x^2)^{-1/2}$ follows in a similar way, although clearly cannot go through the previous claim (since $(1-b\p_x^2)^{-1/2}$ does not have a bounded inverse). To show that $\la_0(\tilde{L})<0$, take an  eigenvector say $g_0:\|g_0\|=1$,
 corresponding to the negative eigenvalue $-a^2$ for $L$. Note that by the first claim, such a $g_0$ is smooth, so in particular $(1-b\p_x^2)^{1/2} g_0$ is well-defined, smooth and non-zero. We have
 $$
 \la_0(\tilde{L})\leq
 \f{\dpr{\tilde{L} (1-b\p_x^2)^{1/2} g_0}{(1-b\p_x^2)^{1/2} g_0}}{\|(1-b\p_x^2)^{1/2} g_0\|^2}=\f{\dpr{L g_0}{g_0}}{\|(1-b\p_x^2)^{1/2} g_0\|^2}=-\f{a_0^2}{\|(1-b\p_x^2)^{1/2} g_0\|^2}<0.
 $$
 Next, to show that $\la_1(\cl)\geq 0$ (the fact that $0$ is an eigenvalue for $\tilde{L}$ was established already), recall that since $L$ has a simple negative eigenvalue, with eigenfunction $g_0$, we have
 $$
 \inf_{g: g\perp g_0} \dpr{L g}{g} = 0.
 $$
 It follows that
 $$
 \la_1(\tilde{L})\geq
 \inf_{f\perp (1-b\p_x^2)^{-1/2} g_0} \f{\dpr{\tilde{L} f}{f}}{\|f\|^2}=
 \inf_{h\perp g_0} \f{\dpr{L h}{h}}{\|(1-b\p_x^2)^{1/2} h\|^2}\geq 0.
 $$
 Regarding the proof of $\la_2(\tilde{L})>0$, we start with $\la_2(\tilde{L})\geq \la_1(\tilde{L})=0$ and we reach a contradiction as before (i.e. we generate two linearly independent vectors in $Ker(L)$), if
 we assume that $\la_2(\tilde{L})=0$.
 \end{proof}
 Using Lemma \ref{le:100}, allows us to
  reduce the proof of \eqref{15} to the proof of
  \begin{equation}
  \label{200}
  \la_1(L)<0=\la_1(L)<\la_2(L),
  \end{equation}
  which we now concentrate on.
  
We have
\begin{eqnarray*}
L &=& \left(\begin{array}{cc}
1+a\p_x^2 & b w \p_x^2 +\psi-w \\
b w \p_x^2 +\psi-w & 1+a \p_x^2 +\vp
\end{array}\right)= \\
&=&  (1+a\p_x^2 ) Id+(b w\p_x^2-w) \left(\begin{array}{cc}
0 & 1 \\
1 & 0
\end{array}\right)+\left(\begin{array}{cc}
0 & \psi \\
\psi & \vp
\end{array}\right)
\end{eqnarray*}
 Introduce an orthogonal matrix $T=\left(
\begin{array}{cc}
 \frac{1}{\sqrt{2}} & \frac{1}{\sqrt{2}} \\
 -\frac{1}{\sqrt{2}} & \frac{1}{\sqrt{2}}
\end{array}
\right)$ and observe that
$$
\left(\begin{array}{cc}
0 & 1 \\
1 & 0
\end{array}\right)=T^{-1} \left(\begin{array}{cc}
1 & 0 \\
0 &-1
\end{array}\right)T.
$$
It follows that
$$
L=T^{-1}\left(   (1+a\p_x^2 ) Id  +(b w\p_x^2-w+\psi)  \left(\begin{array}{cc}
1 & 0 \\
0 &-1
\end{array}\right)   +\f{\vp}{2} \left(\begin{array}{cc}
1 & 1 \\
1& 1
\end{array}\right)   \right)T,
$$
whence, by unitary equivalence, it suffices to consider the operator inside the parentheses.  That is, we consider
\begin{equation}
\label{700}
M=\left(\begin{array}{cc}
-\p_x^2(-a-b w)+(1-w)+\psi+\f{\vp}{2} & \f{\vp}{2} \\
\f{\vp}{2} & -\p_x^2(-a+b w)+(1+w)-\psi+\f{\vp}{2}
\end{array}
\right)
\end{equation}
We shall need the following
\begin{lemma}
\label{le:90}
Let $\al, \la>0$ and $Q\in \rone$. Then, the Hill operator
$$
\cl=-\p_x^2 +\al^2-Q sech^2(\la x)\geq 0
$$
  if and only if
\begin{equation}
\label{230}
\al^2+\al \la \geq Q.
\end{equation}
\end{lemma}
\begin{proof}
This is standard result, which follows from the ones found in the literature by a simple change of variables. First, if $Q\leq 0$, we see right away that $\cl>0$ and also the inequality \eqref{230} is satisfied as well. So, assume $Q>0$. Consider $\cl f = \si f$ and introduce  $f(x)=g(\la x)$. We have (after dividing by $\la^2$ and assigning $y=\la x$)
$$
[-  \p_{yy}+\f{\al^2}{\la^2} - \f{Q}{\la^2} sech^2(y) ] g=\f{\si}{\la^2} g(y)
$$
Recall that the negative the operator $-\p_{yy}-Z sech^2(y)$ are
$k_m=-\left[ \left(
Z+\frac{1}{4}\right)^{\frac{1}{2}}-m-\frac{1}{2}\right]^2$,
provided $\left(
Z+\frac{1}{4}\right)^{\frac{1}{2}}-m-\frac{1}{2}>0$, $m=0,1,2...$
[see \cite{Ab}]. Note that $k_0=\inf \si(-\p_{yy}-Z sech^2(y))$ and hence, to avoid negative spectrum, we need to have
$$
0 \leq \f{\al^2}{\la^2} +k_0= \f{\al^2}{\la^2} - \left[ \left(
\f{Q}{\la^2}+\frac{1}{4}\right)^{\frac{1}{2}}-\frac{1}{2}\right]^2
$$
Solving this last inequality yields \eqref{230}.
\end{proof}
We are now ready to proceed with the count of $n(\tilde{L})$ in  each particular case of consideration. \\
\\
{\bf Case I: $a=c=-b, b>0$} \\
\\
 Going back to the operator $M$, we can rewrite it as 
 $$
 M= S \left(\begin{array}{cc} 
 -\p_x^2+\f{1}{b}+ \f{B+\f{1}{2}}{ b(1-w)}\vp  &  \f{ \vp }{2b\sqrt{1-w^2}} \\
 \f{ \vp }{2b\sqrt{1-w^2}}   &   -\p_x^2+\f{1}{b}+ \f{-B+\f{1}{2}}{ b(1+w)}\vp  
 \end{array}
 \right)    S 
 $$
 where $S=\left(\begin{array}{cc} 
 \sqrt{b(1-w)} & 0 \\
 0 & \sqrt{b(1+w)}
 \end{array}
 \right)$. 
 Thus, according to Lemma \ref{le:100}, we have reduced matters to 
 $$
 M_1=(-\p_x^2+\f{1}{b})Id+ \vp \left(\begin{array}{cc} 
 \f{B+\f{1}{2}}{ b(1-w)}  & \f{ 1 }{2b\sqrt{1-w^2}} \\
  \f{ 1 }{2b\sqrt{1-w^2}} &    \f{-B+\f{1}{2}}{ b(1+w)}
 \end{array}
 \right)
 $$
 Diagonalizing this last symmetric matrix yields the representation 
 $$
 \left(\begin{array}{cc} 
 \f{B+\f{1}{2}}{ b(1-w)}  & \f{ 1 }{2b\sqrt{1-w^2}} \\
  \f{ 1 }{2b\sqrt{1-w^2}} &   \vp \f{-B+\f{1}{2}}{ b(1+w)}
 \end{array}
 \right)=U^* \left(\begin{array}{cc} 
\frac{1+2 B w+\sqrt{4 B^2+4 B w+1} }{2 b \left(1-w^2\right)}  & 0 \\
0 &    \frac{1+2 B w - \sqrt{4 B^2+4 B w+1} }{2 b \left(1-w^2\right)}
 \end{array} \right) U 
 $$
for some orthogonal matrix $U$.  Factoring out $U^*, U$ again and using Lemma \ref{le:100} once more reduces us to the operator
$$
M_2= \left(\begin{array}{cc} 
\cl_1 & 0 \\
0 &   \cl_2 
 \end{array} \right)
$$
 which contains the following  Hill operators on the main diagonal 
 \begin{eqnarray*}
 \cl_1 &=& -\p_x^2+\f{1}{b}+\eta_0 \frac{1+2 B w+\sqrt{4 B^2+4 B w+1} }{2 b \left(1-w^2\right)} sech^2\left(\f{x}{2\sqrt{b}}\right); \\ 
 \cl_2 &=& -\p_x^2+\f{1}{b}+ \eta_0 \frac{1+2 B w - \sqrt{4 B^2+4 B w+1} }{2 b \left(1-w^2\right)} sech^2\left(\f{x}{2\sqrt{b}}\right)
 \end{eqnarray*}
Note that $n(\tilde{L})=n(\cl_1)+n(\cl_2)$.  

Using the formulas 
 $$
 B(\eta_0)=\pm \sqrt{\f{3}{3+\eta_0}}, w(\eta_0)=\pm \f{3+2\eta_0}{\sqrt{3(3+\eta_0)}}
 $$
 yields 
 \begin{eqnarray*}
 \cl_1 &=& -\p_x^2+\f{1}{b}-\f{3}{b} sech^2\left(\f{x}{2\sqrt{b}}\right) ; \\ 
 \cl_2 &=& -\p_x^2+\f{1}{b}-\f{3\eta_0}{b(9+4\eta_0)}sech^2\left(\f{x}{2\sqrt{b}}\right)
 \end{eqnarray*}
 According to the formulas for the eigenvalues in Lemma \ref{le:90} 
 (with $\al=\f{1}{\sqrt{b}}, \la=\f{1}{2\sqrt{b}}$,$ Q=\f{3}{b}>0$) we have that 
 $$
 \la_1(\cl_1)= \f{\al^2}{\la^2}-\left(\sqrt{\f{Q}{\la^2}+\f{1}{4}}-\f{3}{2}\right)^2=2-(\sqrt{12.25}-1.5)^2=0,
 $$
 which indicates that $\cl_1$ has one negative eigenvalue and the next one is zero, whence $n(\cl_1)=1$ for all $\eta_0>-3$.  Thus, $n(\tilde{L})=1+n(\cl_2)$.  
  It is also immediately clear that for $\eta_0\in (-\f{9}{4},0)$, $\cl_2>0$ and hence $n(\tilde{L})=1$.      
 \\
 \\
      {\bf Case II: $a=c<0, b=d>0,$ $a+b\neq 0$} \\
      \\
      In this case, we have $p=\f{c+b}{a+b}=1$, $\eta_0=\f{3(1-2p)}{2p}=-\f{3}{2}$ and thus $w(\eta_0)=w(-3/2)=0$, $\la=\f{1}{2\sqrt{-a}}, B(\eta_0)=\pm \sqrt{2}$.
      This simplifies the computations quite a bit. In fact, starting from the operator $M$, defined in \eqref{700}, we see that it has the form
      $$
      M=(a\p_x^2+1) Id +\left(\begin{array}{cc} B+\f{1}{2} & \f{1}{2} \\
      \f{1}{2} & -B+\f{1}{2}   \end{array}\right)\vp
      $$
      Recall that here $B=\pm \sqrt{2}$. Consider first $B=\sqrt{2}$.
      Diagonalizing the matrix vian an orthogonal matrix $S$ yields the representation
      \begin{eqnarray*}
      & &
      \left(\begin{array}{cc} \sqrt{2}+\f{1}{2} & \f{1}{2} \\
      \f{1}{2} & -\sqrt{2}+\f{1}{2}   \end{array}\right)= S^{-1}
      \left(\begin{array}{cc} 2 & 0 \\
      0 & -1   \end{array}\right) S, \\
      & & S=\f{1}{\sqrt{6}}\left(
\begin{array}{cc}
 \sqrt{3+2 \sqrt{2}} &
   \sqrt{3-2 \sqrt{2}} \\
 -\sqrt{3-2 \sqrt{2}} &
 \sqrt{3+2\sqrt{2}}
\end{array}
\right)
      \end{eqnarray*}
 Thus, in this case, we have represented the operator $L$ in the form
 \begin{equation}
 \label{710}
 L= (S T)^* \left(\begin{array}{cc} -a\p_x^2+1+2\vp & 0 \\
     0 & -a\p_x^2+1-\vp   \end{array}\right)S T,
 \end{equation}
     where $S, T$ are explicit orthogonal matrices. It is now clear that since $\eta_0=-\f{3}{2}<0$, we have that $\vp(x)<0$ and hence the operator
     $a\p_x^2+1-\vp>0$. On the other hand,
     $L_{KdV}= a\p_x^2+1+2\vp$ is well known to have a zero eigenvalue (with eigenfunction $\vp'$) and an unique simple negative eigenvalue.

For the case $B=-\sqrt{2}$, we have \eqref{710}, with
$$
S=\f{1}{\sqrt{6}}\left(
\begin{array}{cc}
 \sqrt{3-2 \sqrt{2}} &
   \sqrt{3+2 \sqrt{2}} \\
 \sqrt{3+2 \sqrt{2}} &
 -\sqrt{3-2\sqrt{2}}
\end{array}
\right)
$$

\section{Proof of Proposition \ref{prop:30}}
\label{sec:4}
The purpose of this section is to compute the quantity appearing in \eqref{150}, whose negativity will be equivalent to the stability of the waves.  Thus, we need to find
$$
L^{-1} [(1-b \p_x^2)\left(\begin{array}{c}  \psi \\
\vp\end{array} \right)].
$$
Here, our considerations need to be split   in two cases:  $a=c=-b$, and $a=c<0, b>0$.   

The case $a=c=-b$ is easier to manage, since in int we have a {\it a free parameter} $w=w(\eta_0)$ that we can differentiate with respect to in \eqref{5}. The remaining case is  harder, because the parameter $\eta_0=-3/2$, whence $w=0$ and one cannot apply the same technique.   
\subsection{The case $a=c=-b$, $b>0$}
Taking a derivative with respect to $w$ in \eqref{5}, we find
$$
L[\left(\begin{array}{c}  \p_w \vp \\
\p_w \psi\end{array} \right)= (1-b \p_x^2)\left(\begin{array}{c}  \psi \\
\vp\end{array} \right),
$$
whence
$$
L^{-1} [(1-b \p_x^2)\left(\begin{array}{c}  \psi \\
\vp\end{array} \right)]=\left(\begin{array}{c}  \p_w \vp \\
\p_w \psi\end{array} \right).
$$
We obtain
\begin{eqnarray*}
& & \dpr{L^{-1}[(1-b \p_x^2) \left(\begin{array}{c}\psi \\ \vp \end{array}\right)]}{(1-b \p_x^2) \left(\begin{array}{c}\psi \\ \vp \end{array}\right)}=  \dpr{(1-b \p_x^2)\left(\begin{array}{c}  \psi \\
\vp\end{array} \right)}{\left(\begin{array}{c}  \p_w \vp \\
\p_w \psi\end{array} \right)}=\\
&=& \p_w [\dpr{\vp}{\psi}+ b  \dpr{\vp'}{\psi'}] =
\p_w[B(\eta_0)\int \vp(\xi)^2+b (\vp'(\xi))^2 d\xi]= \\
&=&
B\p_w [\int_{\mathbb{R}}{[\varphi^2(\xi)+b\varphi'^2(\xi)]}d\xi]+\p_w B  \int_{\mathbb{R}}{[\varphi^2(\xi)+b\varphi'^2(\xi)]}d\xi=\\
&=& \frac{16\sqrt{b}}{5}\left[ B\frac{d\eta_0^2}{dw}+\eta_0^2\frac{dB}{dw}\right]=\frac{16\sqrt{b}}{5}\left[ 2B+\eta_0\frac{dB}{d\eta_0}\right]\eta_0\frac{d\eta_0}{dw}=:d(w)
\end{eqnarray*}
We are now ready to compute this last expression in the cases of interest. \\
\subsubsection{$B(\eta_0)=-\sqrt{\f{3}{3+\eta_0}}, 
w=-\frac{3+2\eta_0}{\sqrt{3(3+\eta_0)}}$} 
   We have
    $$
     \frac{d\eta_0}{dw}=-\frac{2\sqrt{3}(3+\eta_0)^{\frac{3}{2}}}{2\eta_0+9},
     \ \ \frac{dB}{d\eta_0}=\frac{\sqrt{3}}{2}\frac{1}{(3+\eta_0)^{\frac{3}{2}}}
     $$ 
      and
      $$d(w)=-\frac{48\sqrt{3b}}{10(3+\eta_0)^{\frac{3}{2}}}(4+\eta_0)\eta_0\frac{d\eta_0}{dw}<0
      $$ 
  for $-\f{9}{4}<\eta_0<0$. 
  \subsubsection{$B(\eta_0)=\sqrt{\f{3}{3+\eta_0}}, w=\frac{3+2\eta_0}{\sqrt{3(3+\eta_0)}}$}     We have 
 $$
     \frac{d\eta_0}{dw}=\frac{2\sqrt{3}(3+\eta_0)^{\frac{3}{2}}}{2\eta_0+9},
     \ \ \frac{dB}{d\eta_0}=-\frac{\sqrt{3}}{2}\frac{1}{(3+\eta_0)^{\frac{3}{2}}}
     $$
     hence 
     $$
     d(w)=\frac{48\sqrt{3b}}{10(3+\eta_0)^{\frac{3}{2}}}(4+\eta_0)\eta_0\frac{d\eta_0}{dw}<0. 
      $$ 
   for $-\f{9}{4}<\eta_0<0$.  
      
\subsection{The  case: $a=c<0, b>0$} As we have discussed above, we have   explicit formulas for all the quantities involved. Namely,
we have  $w=0, \la=\f{1}{2\sqrt{-a}}, B =\pm \sqrt{2}$. Thus, 
$$
\vp(x)=-\f{3}{2} sech^2\left(\f{x}{2\sqrt{-a}}\right).
$$
\subsubsection{Case $B=\sqrt{2}$}
We need to compute
$$
\dpr{L^{-1}\left( \begin{array}{c} (1-b\p_x^2)\psi \\ (1-b\p_x^2)\vp \end{array} \right)}{\left( \begin{array}{c} (1-b\p_x^2)\psi \\ (1-b\p_x^2)\vp \end{array} \right)}
$$
 To that end, we use the representation \eqref{710}. We have
 \begin{eqnarray*}
 I &=& \dpr{L^{-1}\left( \begin{array}{c} (1-b\p_x^2)\psi \\ (1-b\p_x^2)\vp \end{array} \right)}{\left( \begin{array}{c} (1-b\p_x^2)\psi \\ (1-b\p_x^2)\vp \end{array} \right)} = \\
 & & = \dpr{\left(\begin{array}{cc} a\p_x^2+1+2\vp & 0 \\
     0 & a\p_x^2+1-\vp   \end{array}\right)[ S T \left( \begin{array}{c} \sqrt{2} \\ 1 \end{array} \right)(1-b\p_x^2)\vp]}{S T \left( \begin{array}{c} \sqrt{2} \\ 1 \end{array} \right)(1-b\p_x^2)\vp }
 \end{eqnarray*}
A direct computation shows that $S T \left( \begin{array}{c} \sqrt{2} \\ 1 \end{array} \right)=\left(\begin{array}{c} 2 \sqrt{\frac{2}{3}} \\  -\frac{1}{\sqrt{3}}\end{array}\right)$, whence our index $I$ can be computed as follows
$$
I=\f{8}{3}\dpr{(a\p_x^2+1+2\vp)^{-1}[(1-b\p_x^2)\vp]}{(1-b\p_x^2)\vp}+
\f{1}{3}\dpr{(a\p_x^2+1-\vp)^{-1}[(1-b\p_x^2)\vp]}{(1-b\p_x^2)\vp}
$$
Denote $f=(1-b\p_x^2)\vp$ and 
\begin{eqnarray*}
L_{KdV} &=& a\p_x^2+1+2\vp \\ 
L_{Hill} &=& a\p_x^2+1-\vp
\end{eqnarray*} 
 Note that by Weyl's theorem 
  $ \sigma_{ess.}(L_{Hill})=[1, \infty)$. On the other hand, by the fact that $\vp<0$, the potential $-\vp>0$ and hence, by the results for 
  absence of embedded eigenvalues, $\sigma(L_{Hill})=\sigma_{ess.}(L_{Hill})=[1, \infty)$. 
   We now  compute the index 
$$
I=\f{1}{3}(8 \dpr{L_{KdV}^{-1} f}{f}+\dpr{L_{Hill}^{-1} f}{f}). 
$$
To that end, we differentiate  the equation 
  $$
  a\varphi''+\varphi+\varphi^2=0 
  $$
  with respect to $a$. We get\footnote{we use the notation $\varphi_a=\p_a \vp$ denotes the derivative with respect to $a$}
  \begin{equation}
  \label{720a}
    L_{KdV}\varphi_a=-\varphi'',
  \end{equation}
whence  $L_{KdV}^{-1}[\vp'']=-\vp_a$. 
  Using that $L_{KdV}\varphi=\varphi^2=-a\vp''-\vp$ and the above relation, we obtain that
    $$
   -\varphi=aL_{KdV}^{-1}\varphi''+L_{KdV}^{-1}\varphi=-a\vp_a+L_{KdV}^{-1}\varphi . 
    $$
It follows that 
    \begin{eqnarray}
    \label{720b}
      L_{KdV}^{-1}\varphi &=& a\varphi_a-\varphi, \\
     \label{720c}
        L_{KdV}^{-1}f &=& (a+b)\varphi_a-\varphi .
      \end{eqnarray}
      and
       $$
          \langle L_{KdV}^{-1}f,f\rangle=(a+b)\langle \varphi_a, \varphi
          \rangle-b(a+b)\langle \varphi_a, \varphi'' \rangle-\langle \varphi, \varphi
          \rangle+b\langle \varphi, \varphi'' \rangle.
    $$
      By direct computations
          \begin{eqnarray*}
          \langle \varphi_a, \varphi
          \rangle &=& \frac{1}{2}\frac{d}{da}\int_{-\infty}^{+\infty}{\varphi^2}dx=
          -\f{3}{2\sqrt{-a}}, \\
          \langle \varphi, \varphi''
          \rangle &=& -\int_{-\infty}^{+\infty}{\varphi'^2}dx=-\frac{6}{5\sqrt{-a}}, \\ 
         \langle \varphi_a, \varphi''
          \rangle &=& -\frac{1}{2}\frac{d}{da}\int_{-\infty}^{+\infty}{\varphi'^2}dx=-\frac{3}{10|a|\sqrt{-a}}, \\
          \langle \varphi, \varphi
          \rangle &=& \f{9}{2}\sqrt{-a}
          \int_{-\infty}^{+\infty}{sech^4(y)}dy=6 \sqrt{-a}. 
          \end{eqnarray*} 
As a consequence, 
  \begin{eqnarray*}
  \langle
  L_{KdV}^{-1}f,f\rangle &=& -\frac{3(a+b)}{2\sqrt{-a}}+\frac{3b(a+b)}{10|a|
  \sqrt{-a}}-6\sqrt{-a}-\frac{6b}{5\sqrt{-a}}=\\
  &=& -\f{9}{2}\sqrt{-a}-\f{12}{5}\f{b}{\sqrt{-a}} + 
  \f{3b^2}{10|a| \sqrt{-a}}=   \sqrt{-a}
  \left( -\f{9}{2}-\f{12}{5} \f{b}{|a|}+ 
 \f{3}{10} \f{b^2}{a^2}\right). 
  \end{eqnarray*}
  
This yields the desired computation for the terms involving $L_{KdV}^{-1}$. We   turn our attention to $L_{Hill}^{-1}$. The situation here is a bit trickier, since we cannot compute explicitly the quantities $L_{Hill}^{-1}[\vp], L_{Hill}^{-1}[\vp'']$, as required in the formula for $I$. Instead, we need to rely on   estimates. To start with, observe that 
$$
L_{Hill}[\vp]=a\vp''+\vp-\vp^2=-2\vp^2=2a\vp''+2\vp,
$$
whence 
\begin{equation}
\label{800}
L_{Hill}^{-1}[a\vp''+\vp]=\f{\vp}{2}.  
\end{equation}
    Since we need to compute $ L_{Hill}^{-1}[f]=L_{Hill}^{-1}[\vp-b\vp'']$, we project the vector $f$ onto $a\vp''+\vp$ and its orthogonal subspace as follows 
    $$
    f=\vp-b\vp''=\f{\dpr{\vp-b \vp''}{a\vp''+\vp}}{\|a\vp''+\vp\|^2}(a\vp''+\vp)+g
    $$
    Calculations then show that since 
    $$
   \|\vp''\|^2= \dpr{\vp''}{\vp''}=\f{6}{7 |a|\sqrt{-a}}, 
    $$
    we have that 
    $$
    f=\left(\f{7}{9}+\f{2}{9}\f{b}{|a|}\right) (a\vp''+\vp)+g
    $$
    whence 
    $$
    L_{Hill}^{-1}[f]=\left(\f{7}{9}+\f{2}{9}\f{b}{|a|}\right) \f{\vp}{2}+L_{Hill}^{-1}[g].
    $$
    Thus, the quantity that needs to be computed is 
    \begin{eqnarray*}
 \dpr{ L_{Hill}^{-1} f}{f} &=& 
    \f{1}{2}\left(\f{7}{9}+\f{2}{9}\f{b}{|a|}\right)\dpr{\vp-b\vp''}{\vp}
    + \dpr{ L_{Hill}^{-1} g}{f}=\\
    &=& \f{1}{2}\left(\f{7}{9}+\f{2}{9}\f{b}{|a|}\right)\dpr{\vp-b\vp''}{\vp} + 
    \dpr{ L_{Hill}^{-1} g}{g}+\f{1}{2}\left(\f{7}{9}+\f{2}{9}\f{b}{|a|}\right)\dpr{g}{\vp}
    \end{eqnarray*}
All of these can be computed explicitly, except for $  \dpr{ L_{Hill}^{-1} g}{g}$, which we estimate by \\  $  0<\dpr{ L_{Hill}^{-1} g}{g}\leq \|g\|^2$, which holds since   $\si(L_{Hill})\subset [1, \infty)$. Thus, 
\begin{eqnarray*}
 \dpr{ L_{Hill}^{-1} f}{f} &\leq & \f{1}{2}\left(\f{7}{9}+\f{2}{9}\f{b}{|a|}\right)\dpr{\vp-b\vp''}{\vp} +  \|g\|^2 +\f{1}{2}\left(\f{7}{9}+\f{2}{9}\f{b}{|a|}\right)\dpr{g}{\vp}=\\
 & = &  \sqrt{-a}\left( \f{22}{45}\f{b^2}{a^2}+\f{2}{9} \f{b}{|a|}+\f{26}{9}\right)
 \end{eqnarray*}
 and on the other hand 
 $$
  \dpr{ L_{Hill}^{-1} f}{f}>\f{1}{2}\left(\f{7}{9}+\f{2}{9}\f{b}{|a|}\right)\dpr{\vp-b\vp''}{\vp} +   \f{1}{2}\left(\f{7}{9}+\f{2}{9}\f{b}{|a|}\right)\dpr{g}{\vp}=
  \sqrt{-a}\left( \f{4}{45} \f{b^2}{a^2} + \f{46}{45} \f{b}{|a| } +\f{ 112}{45}  \right). 
 $$
  Thus, we obtain the following {\it estimate} for the instability index $I$
 \begin{eqnarray*}
  3 I &=& 8 \dpr{L_{KdV}^{-1} f}{f}+\dpr{L_{Hill}^{-1} f}{f}\leq \sqrt{-a}\left(   8\left(-\f{9}{2}-\f{12 b}{5|a|}+ 
 \f{3 b^2}{10 a^2}\right) 
 +  \left( \f{22}{45}\f{b^2}{a^2}+\f{2}{9} \f{b}{\sqrt{-a}}+\f{26}{9}\right)\right)=\\
 &=& \f{2\sqrt{-a}}{45}\left( 65 \f{b^2}{a^2}-427 \f{b}{\sqrt{-a}}-745\right)    
  \end{eqnarray*}

On the other hand, we have the following estimate from below 
\begin{eqnarray*}
  3 I &=& 8 \dpr{L_{KdV}^{-1} f}{f}+\dpr{L_{Hill}^{-1} f}{f} > \sqrt{-a} \left( 8\left(-\f{9}{2}-\f{12 b}{5|a|}+ 
 \f{3 b^2}{10 a^2}\right)+  \left( \f{4}{45} \f{b^2}{a^2} + \f{46}{45} \f{b}{|a| } +\f{ 112}{45}  \right)  \right)\\
 &=&  \f{2\sqrt{-a}}{45}\left(56 \f{b^2}{a^2}-409  \f{b}{|a| } -754    \right).
  \end{eqnarray*}
The picture below shows the graphs of the two estimates of $3I/\sqrt{-a}$. If one solves the corresponding quadratic equations, we see that  we have {\bf stability}, whenever 
$$
0\leq \f{b}{-a}< \frac{1}{130} \left(427+3 \sqrt{41781}\right)\sim 8.00163. 
$$
and {\bf instability}, when 
$$
\f{b}{-a}> \frac{1}{112} \left(409+3 \sqrt{37353}\right)\sim 8.82864. 
$$
   \begin{figure}[h]
\centering
\includegraphics[width=15cm,height=7cm]{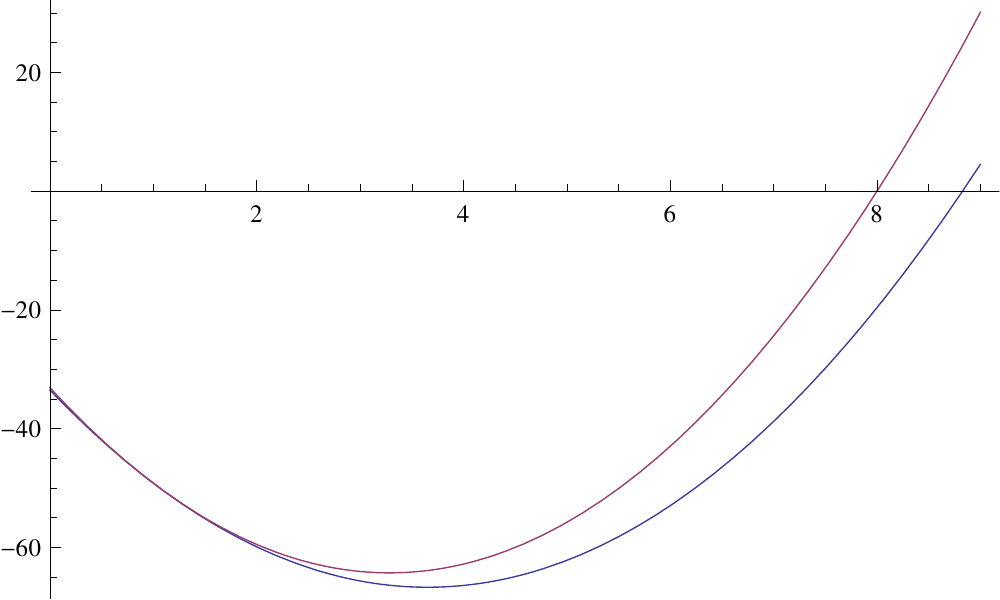}
\caption{The picture shows the graphs of the function 
$\f{2}{45}\left( 65 z^2-427 z-745\right) $ is in blue, while 
$ \f{2}{45} (56 z^2-409  z -754) $ in red. Note that the graphs do coincide for $z=1$, which is the case, since there $g=0$ and the computations becomes  precise. }
\label{fig}
\end{figure}

\subsubsection{Case $B=-\sqrt{2}$}
In this case, the computation for the index is the same since
\begin{eqnarray*}
 I &=& \dpr{L^{-1}\left( \begin{array}{c} (1-b\p_x^2)\psi \\ (1-b\p_x^2)\vp \end{array} \right)}{\left( \begin{array}{c} (1-b\p_x^2)\psi \\ (1-b\p_x^2)\vp \end{array} \right)} = \\
 &=&  \dpr{\left(\begin{array}{cc} a\p_x^2+1+2\vp & 0 \\
     0 & a\p_x^2+1-\vp   \end{array}\right)[ S T \left( \begin{array}{c} -\sqrt{2} \\ 1 \end{array} \right)(1-b\p_x^2)\vp]}{S T \left( \begin{array}{c} -\sqrt{2} \\ 1 \end{array} \right)(1-b\p_x^2)\vp }=\\
     &=& \f{1}{3}(8 \dpr{L_{KdV}^{-1} f}{f}+\dpr{L_{Hill}^{-1} f}{f}),
 \end{eqnarray*}
 where in the last line, we have used that $S T \left( \begin{array}{c} \sqrt{2} \\ 1 \end{array} \right)=\left(\begin{array}{c} 2 \sqrt{\frac{2}{3}} \\  -\frac{1}{\sqrt{3}}\end{array}\right)$ as above. The rest of the argument proceeds in exactly the same way, since the exact same quantity is being computed.

\end{document}